\documentclass[a4paper]{amsart}

\def\AA{{\mathbb{A}}}
\def\PP{{\mathbb{P}}}

\def\CC{{\mathbb{C}}}

\def\NN{{\mathbb{N}}}

\let \cedilla =\c

\newcommand{\p}{{\frak{p}}}
\newcommand{\q}{{\frak{q}}}

\newcommand{\oo}{{\mathcal{O}}}
\newcommand{\spec}{{\mathrm{Spec \ }}}
\newcommand{\proj}{{\mathrm{Proj\ }}}
\newcommand{\cha}{{\operatorname{char}}}
\newcommand{\m}{{\frak{m}}}
\newcommand{\supp}{{\mathrm{supp}}}
\newcommand{\height}{{\mathrm{ht} }}

\newcommand{\ch}{{\rm char}}
 \newcommand{\R}{{\mathcal{R}}}
 \newcommand{\depth}{{\mathrm{depth}}}
 
 \newcommand{\grade}{{\mathrm{grade}}}

\newcommand{\M}{{\mathcal{M}}}

\newtheorem{thm}{Theorem}[section]
\newtheorem{cor}[thm]{Corollary}
\newtheorem{Corollary-Definition}[thm]{Corollary-Definition}
\newtheorem{prop-def}[thm]{Proposition-Definition}

\newtheorem{lem}[thm]{Lemma}

\theoremstyle{definition}
\newtheorem{defn}[thm]{Definition}
\newtheorem{ex}[thm]{Example}
\newtheorem{rem}[thm]{Remark}

\newtheorem{Proposition-Definition}[thm]{Proposition-Definition}

\begin{document}

\title[Higher direct images]{On vanishing of higher direct images of the structure sheaf}
\author[Ishii]{Shihoko Ishii}
\address[Shihoko Ishii]{Graduate School of Mathematical Science, the University of Tokyo,
Meguro-ku, Tokyo, 153-8914, Japan}
\email{shihoko@g.ecc.u-tokyo.ac.jp}

\author[Yoshida]{Ken-ichi Yoshida}
\address[Ken-ichi Yoshida]{Department of Mathematics, 
College of Humanities and Sciences, 
Nihon University, Setagaya-ku, Tokyo, 156-8550, Japan}
\email{yoshida.kennichi@nihon-u.ac.jp}

\dedicatory{Dedicated to Professor Kei-ichi Watanabe on the occasion of his 80th birthday}

\thanks{
The authors are partially supported by Grant-In-Aid (c) 22K03428 and 24K06678  of JSPS, 
and also by the Joint Usage Program of Osaka Central Advanced Mathematical Institute (OCAMI).
 }

\maketitle

\begin{abstract}
\noindent
We show the vanishing of the first direct image of the structure sheaf of a normal  scheme $X$
which is  mapped  properly and birationally over an regular scheme of any dimension.
On the other hand, for any dimension greater than two,
we show examples of a proper birational morphism from a normal and Cohen-Macaulay scheme to a regular scheme
such that the second direct image  does not vanish and has an isolated support.
\end{abstract}

\section{Introduction}
\noindent
Let $\varphi: X \to Y:=\spec R$ be a projective birational morphism of schemes.
It is well known that if $X$ and $Y$ are regular schemes of finite type over a field of characteristic $0$ 
the following holds:
$$R^i\varphi_*\oo_X=0\ \mbox{  for\ every} \ i>0$$
by \cite[pp.144-145]{hir}.
This is generalized by \cite[Theorem 1.1]{cr} for excellent regular schemes $X, Y$.

To illustrate some history, we temporarily assume that $Y$ is a variety over a field $k$ of characteristic $0$  with singularities.
Let $\varphi: X \to Y$ be a resolution of the singularities of $Y$.
Then $R^i\varphi_*\oo_X$ ($i\in \NN$) play important roles in measuring how far the singularities of $Y$ are from 
regularity since these are zero if $Y$ is regular.
The closest singularities to a regular point in this viewpoint are rational singularities defined as follows:

\begin{defn} Assuming $\cha k=0$, we say that $Y$ has {\it rational singularities} if for every resolution
$\varphi: X \to Y$ of the singularities of $Y$, the following holds:
\begin{equation}\label{vani}
 R^i\varphi_*\oo_X=0\ \mbox{for\ every\ } i>0.
\end{equation}
\end{defn}

Here, note that once the vanishings above hold for a resolution $\varphi$, then the vanishings also hold for every resolution of the
singularities of $Y$. 
It is also notable that a rational singularity automatically has the Cohen-Macaulay (CM for short) property,
which is considered an essential condition for a singularity to be ``good''.
On the other hand, in positive characteristic, the vanishing (\ref{vani}) does not automatically imply CM property.
Based on this and the present situation, the existence of resolutions of singularities in positive characteristic 
is not yet proved,
there is a challenge reasonably defining a rational singularity without using a resolution of the singularities as follows:

\begin{defn}[{\cite[Definition 1.3]{kov}}]\label{kov} A scheme $Y$ is said to have rational singularities, if
\begin{enumerate}
\item[(i)] Y is an excellent normal Cohen-Macaulay scheme that admits a dualizing complex, and
\item[(ii)] for each excellent normal Cohen-Macaulay scheme $X$, and each $\varphi : X \to  Y$ locally projective
birational morphism, the following vanishing holds:
$$ R^i\varphi_*\oo_X=0\ \mbox{for\ every\ } i>0.$$
\end {enumerate}
\end{defn} 

\noindent
From now on, we will distinguish between rational singularities in $\ch k=0$ and those in the definition \ref{kov} by denoting them as rational singularities and ``rational singularities'',  respectively.
Definition \ref{kov} asserts  that   $\varphi: X \to Y$ plays the role of a resolution of the singularities to
characterize ``rational singularities''.
If the definition of ``rational singularity'' is reasonable, in principle, a regular point should be ``rational'' (cf. \cite[Theorem 9.12]{kov}).
However,
 Linquan Ma shows an example of a regular point that does not satisfy this definition of ``rational singularity".
More concretely,
there is a projective birational morphism 
 $\varphi: X \to \AA^3_\CC,$ from normal and CM variety $X$, such that $ R^2\varphi_*\oo_X\neq0$
 (see Example \ref{ma}).
 It means that $\varphi: X\to Y$ in Definition \ref{kov} does not play the role of a resolution of the singularities of $Y$ to characterize 
 rational singularities.

 Here, note that in this example, the following vanishing still holds:
\begin{equation}\label{cm}
  \ R^i\varphi_*\oo_X=0\ \mbox{ holds\  for\  every\ } 0< i <N-1,
\end{equation}  
 where $N=\dim Y$.
 Therefore, one may still expect that 
$\varphi: X \to Y$ with normal CM variety $X$ would play the role of a resolution of the singularities to
characterize an isolated CM singularity.
Here, remember that if $\varphi: X \to Y$ is a resolution of $N$-dimensional isolated  singularity $(Y, 0)$ (therefore $R^i\varphi_*\oo_X |_{Y\setminus\{0\}}=0$
 $(i>0)$ ),
 then, 
 the singularity $(Y,0)$ is CM if and only if 
 $ \ R^i\varphi_*\oo_X=0\ \mbox{ holds\  for\  every\ } 0< i <N-1.$
 
 \vskip.5truecm
 \noindent
 {\bf Question.}
 For a  CM closed point $0\in Y$ on an $N$-dimensional scheme ($N\geq3$) and a proper birational
 morphism $\varphi: X\to Y$ from a normal CM scheme $X$,
 if $R^i\varphi_*\oo_X |_{Y\setminus\{0\}}=0$
 $(i>0)$ holds, then does
  the following vanishings hold?
 \begin{equation}\label{0<i<N-1}
 \ R^i\varphi_*\oo_X=0\ \mbox{ holds\  for\  every\ } 0< i <N-1.
 \end{equation} 
 

In this paper, we demonstrate that this question is only positively answered in a specific case.
Concretely, 
 if $i=1$, the vanishing holds in a stronger form (Theorem \ref{R!}), and as its application we obtain the vanishing 
 $$ \ R^1\varphi_*\oo_X=0$$
 for regular $Y$ (Corollary \ref{R1}). 
  Here, we do not need CM for $X$.
 If $i= 2$, the answer to the question is ``NO'' even for smooth $Y$ and normal CM variety $X$ of any dimension $N\geq3$.
We summarize our results in the following.
A local ring $(R,\m)$ is called {\it equidimensional} 
if $\dim R/\mathfrak{p}=\dim R$ holds 
for every minimal prime ideal $\mathfrak{p}$ of $R$. 
Let $\widehat{R}$ denote the $\m$-adic completion of $R$. 
The ring $R$ is called {\it quasi-unmixed} 
if $\widehat{R}$ is equidimensional. 
Therefore, if $\widehat{R}$ is an integral domain, the local ring $R$ is
quasi-unmixed. 
For example, a regular local ring $R$ is quasi-unmixed,
because it is well known that
the completion $\widehat{R}$ is also
a regular local ring (see, eg., \cite[Theorem 23.7]{mats}), 
in particular $\widehat{R}$ is an integral domain.

\par 
A Noetherian ring $R$ is called \textit{locally quasi-unmixed}
if every $\mathfrak{p} \in \spec R$, $R_\mathfrak{p}$ is a quasi-unmixed local ring.

\begin{thm}\label{R!} 
Let $Y$ be an  Noetherian scheme of dimension $N\geq1$.
Assume $Y$ is locally quasi-unmixed and has pseudo-rational singularities in codimension two and the Serre's condition $S_3$. 
 Let $\varphi: X\to Y$  be a proper birational morphism of finite type with normal $X$. 
 Then, the following vanishing holds,
 \begin{equation}\label{0<i<N-1}
 \ R^1\varphi_*\oo_X=0.
 \end{equation}
\end{thm} 
The following corollary follows from Theorem \ref{R!}.
\begin{cor}\label{R1} 
Let $Y$ be a  regular scheme of dimension $N\geq1$.
 Let $\varphi: X\to Y$ be a proper birational morphism of finite type
 with  normal $X$.
  Then $$R^1\varphi_*\oo_X=0.$$
\end{cor} 
Here, note that we do not assume CM on $X$.

  \begin{rem} The paper   \cite[Theorem 1]{lodh} proves a result similar to that in Theorem \ref{R!}. 
  Our condition on \(Y\) for vanishing is slightly less restrictive than that of \cite{lodh}. 
  Specifically, the key difference between the conditions of these theorems is that we assume \(Y\) is quasi-unmixed, whereas \cite{lodh} assumes \(Y\) is quasi-excellent.
\end{rem}

\begin{ex}\label{R2}
Let $(R, \m)$ be a regular local ring of dimension $N\geq3$ essentially of finite type over $\CC$.
There exists a projective birational morphism
  $\varphi: X\to Y:=\spec R$ with a normal and CM variety $X$ such that
 $$R^2\varphi_*\oo_X\neq0, \  \  \mbox{and}\  \   R^i\varphi_*\oo_X|_{Y\setminus\{\m\}}=0
 \ \ (i>0)
$$

\end{ex}

This paper is organized as follows: in the second section we prove the theorem and the third section
we prove the example.

\vskip.5truecm
\noindent
{\bf Acknowledgement} The authors express their heartfelt thanks to S\'andor Kov\'acs, Linquan Ma, Mircea Musta\cedilla{t}\v{a}, 
Karl Schwede, Kohsuke Shibata, Shunsuke Takagi for useful informations and discussions, in particular Linquan Ma's kindness  to allow us
including  
his example (Example \ref{ma}) in this paper.
We also thank Kay R\"ulling for informing us about the paper \cite{lodh}, and the anonymous referee for
constructive comments.

\section{Vanishing of $R^1$}

 \noindent
 In this section, we concentrate to prove
  Theorem \ref{R!}.
  We use the same symbols as in Theorem \ref{R!}.
 Let $\varphi:X\to Y$ be as in the theorem.
 As our problem is local, we may assume that $Y=\spec R$, where $(R,\m)$ is an quasi-umixed local ring satisfying 
 pseudo-rational in codimension two and $S_3$.
 As $Y$ is affine, we sometimes write $H^i(X, \oo_X)$ for $R^i\varphi_*\oo_X$.
 The following lemma will allow us to assume that $\varphi$ is a projective birational morphism.
 
  \begin{lem}\label{composite} Let  $\varphi:X \to Y$ be a proper birational morphism with normal $X$.
  Assume there exists a proper birational morphism $\psi:X'\to X$ such that $R^1(\varphi\circ\psi)_*\oo_{X'}=0$.
  Then, $$R^1\varphi_*\oo_X=0.$$

 \end{lem}
 \begin{proof} 
Consider the Leray spectral sequence 
  $$E_2^{p,q}=R^p\varphi_*(R^q\psi_*\oo_{X'}) \Rightarrow E^{p+q}=R^{p+q}(\varphi\circ\psi)_*\oo_{X'}$$
induces
the exact sequence 
$$0\to R^1\varphi_*(\psi_*\oo_{X'})\to R^1(\varphi\circ\psi)_*\oo_{X'}\to \varphi_*(R^1\psi_*\oo_{X'})\to.$$
As the middle term is zero, the first term becomes zero too, which implies 
 $R^1\varphi_*\oo_{X}=0$ by the normality of $X$. 
 \end{proof}
 
 In our situation of Theorem \ref{R!},  apply 
 Chow's lemma (\cite[II, Theorem 5.6.1]{ega}) to the given proper birational morphism $\varphi:X\to Y$, 
 we obtain a birational morphism $\psi:X'\to X$ such that 
 the composite $\Phi=\varphi\circ\psi:X'\to Y$ is projective birational.
 By Lemma \ref{composite}, for the proof of $R^1\varphi_*\oo_X=0$ it is sufficient to show $R^1\Phi_*\oo_{X'}=0$.
 Now we may assume that $\varphi:X\to Y=\spec R$ is projective birational and $X$ is normal.
Then, there is an ideal $I\subset R$
 whose blow-up gives $\varphi$.
 Then, we can see that $\height I >0$.
 
 Denote the Rees algebra of $I$ by 
 $$\R(I):=R\oplus IT\oplus I^2T^2\oplus\cdots,$$
 where $T$ is an indeterminate,
 and denote the fiber algebra by 
 $$G(I):=R/I \oplus( I/I^2)T\oplus (I^2/I^3)T^2\oplus \cdots.$$
 We denote the maximal ideal  of $\R(I)$ by
 $$\M:=\m\oplus IT\oplus I^2T^2\oplus\cdots.$$
 Note that $$X=\proj \R(I)\simeq \proj\R(I^n),$$
 for every $n\in \NN$.
 An ideal $I\subset R$ is called a normal ideal if  every $I^n$ $(n\in \NN)$ is integrally closed.
 We know that $I\subset R$ is normal if and only if $\R(I)$ is a normal ring.

 We will interpret our vanishing problem  into the depth of the Rees algebra in the following
 lemma.

 \begin{lem}\label{depth}  
 Let $(R,\m)$ be a quasi-unmixed local ring with 
 $\depth R \ge 3$. 
Let $I \subset R$ be an ideal with $\grade I > 0$.   
Suppose that $X=\proj \mathcal{R}(I)$ is normal. 
Then, for $n \gg 0$ the following holds$:$
\[
\depth_{\mathcal{M}} \mathcal{R}(I^n) \ge 3 
\] 
i.e., $H_{\mathcal{M}}^i(\mathcal{R}(I^n))=0$, for $i \le 2$. 
\end{lem}

\begin{proof}
Since $X=\proj \mathcal{R}(I)$ is normal, 
all large powers $I^n$ are normal by \cite[Proposition 3.6]{HM99}. 
So we may assume that $I$ is normal. 
This implies that $\depth_{{\mathcal{M}}} G(I^n) \ge 2$ for large $n \gg 0$ 
from \cite[Theorem 3.4]{HM99} (cf. \cite[Theorem 3.1]{hh}). 
\par \vspace{1mm}
Fix an integer $n$ such that  $\depth_{{\mathcal{M}}} G(I^n) \ge 2$.  
If  $\depth_{\mathcal{M}} G(I^n) \ge 3$, then 
\[
\depth_{{ \mathcal{M}}} \mathcal{R}(I^n) 
\ge \depth_{{ \mathcal{M}}} G(I^n)  \ge 3
\]
by \cite[Proposition 3.6]{HM94}. 
Otherwise, $\depth_{{ \mathcal{M}}} G(I^n) 
{=2 < \depth R}$.  
It follows from \cite[Theorem 3.10]{HM94} that 
\[
\depth_{{ \mathcal{M}}} \mathcal{R}(I^n)=\depth_{{ \mathcal{M}}} G(I^n)+1  = 2+1=3. 
\] 
\end{proof}

 
 If $N=\dim Y=1$, then the vanishing $R^1\varphi_*\oo_X=0$ is trivial. 
 
 If $N=2$, as $Y$ is pseudo-rational, by the definition \cite[pp. 156]{lip78}
 there exists a proper birational morphism $\psi: X'\to X$ such that 
  $R^1(\varphi\circ\psi)_*\oo_{X'}=0$.
 Then,  by Lemma \ref{composite} it follows
 $$R^1\varphi_*\oo_X=0.$$

 For $N\geq3$, note that we can apply Lemma \ref{depth} to our case in Theorem \ref{R!}.
 As  $Y=\spec R$ is the surjective image of normal scheme $X=\proj \R(I)$,
 $R$ is integral domain, and therefore
  $\height I>0$ implies $\grade I>0$.
 Suppose that 
 $$R^1\varphi_*\oo_X\neq 0$$
  and let $C\subset Y$ be an irreducible component of $\supp R^1\varphi_*\oo_X$.
 Let $\p\subset R$ be the ideal defining $C$,
 then,  $\dim R_\p\geq 3$ by the discussions for $N\leq 2$.
 We also have $\depth R_\p\geq 3$ as we assume $S_3$ for $R$.
 Hence, we can apply Lemma \ref{depth} to $(R_\p, \p R_\p)$.
    In this setting we will prove: 
    $$H^1(\varphi^{-1}(\spec R_\p), \oo_{\spec R_\p})=0,$$
    which is a contradiction to the definition of $\p$, and therefore, completes the proof of the theorem.
  Denote $R_\p$ by $R$  and denote its maximal ideal by $\m$.
  Replace $Y$ and $X$ accordingly.
     Let $E:=\varphi^{-1}(\m)$, and $$\varphi':X\setminus E\to Y\setminus\{\m\}$$ 
     the restriction of $\varphi:X\to Y$.
  Then, by the construction of $R$, we have the following vanishing 
  \begin{equation}\label{restriction}
  R^1\varphi'_*\oo_ {X\setminus E}=0. 
\end{equation}
 

For the proof of the theorem, we show the following isomorphism first:
 \begin{equation}\label{isom}
 H^1(X,\oo_{X})\simeq H^1_{E}(X,\oo_{X}).
 \end{equation}
 For (\ref{isom}), consider the exact sequence:
 $$H^0(X,\oo_{X})\stackrel{\alpha}\longrightarrow H^0(X\setminus E,\oo_{X})\to H^1_{E}(X,\oo_{X})
\to H^1(X,\oo_{X})\to H^1(X\setminus E,\oo_{X}).$$
 For the required isomorphism, it is sufficient to prove the following  (I) and (II):
 \begin{center}
  (I) $H^1(X\setminus E,\oo_{X})=0$, and (II) $\alpha $ is surjective. 
\end{center}

Indeed, for the  morphism 
$ \varphi' :X\setminus E \to Y\setminus \{\m\}$
we have the spectral sequence

$$E_2^{p,q}=H^p(Y\setminus\{\m\}, R^q\varphi'_*(\oo_{X\setminus E}))\Rightarrow E^{p+q}=H^{p+q}(X\setminus E,\oo_{X}).$$
In the canonical exact sequence 
$$0\to E_2^{1,0}\to E^1\to E_2^{0,1},$$
$E_2^{1,0}=H^1(Y\setminus\{\m\}, \oo_{Y})=H_\m^2(R)=0$ since  $\depth R\geq3$,
while $E_2^{0,1}=H^0(Y\setminus \{\m\}, R^1\varphi'_*\oo_{X\setminus E})=0$ by (\ref{restriction}).
This gives the vanishing $E^1=H^1(X\setminus E,\oo_{X})=0$ as claimed in (I).

On the other hand,  the above spectral sequence satisfies $E_2^{p,q}=0$ for $p<0$ or $q<0$, which gives the isomorphism
$E^0=E_2^{0,0}$.
Then, the homomorphism $$\alpha:H^0(X,\oo_{X})\to H^0(X\setminus E,\oo_{X})=E^0$$ is the same as
$$H^0(Y,\oo_{Y})\to H^0(Y\setminus\{\m\}, \oo_{Y})=E_2^{0,0}.$$   
Hence, $\alpha$ is isomorphic because of  $H_{\m}^i(R)=0$
for $i=0,1$. Q.E.D. of  (II).

Next, consider the (SdS)-sequence
$$\to H_{\m}^1(R)\to H_{E}^1(X,\oo_{X})\to H_\M^2(\R(I^n))_0\to $$
where $H_\M^2(\R(I^n))_0$ is the degree 0 part of the graded module $H_\M^2(\R(I^n))$.
Here, $H_{\m}^1(R)=0$ follows from $\depth R\geq 3$.
On the other hand, 
$H_\M^2(\R(I^n))_0=0$ for $n\gg0$ follows from Lemma \ref{depth}.
Then, we obtain $H_{E}^1(X,\oo_{X})=0$ which yields $H^1(X,\oo_{X})=0$ by (\ref{isom}).
[QED of Theorem \ref{R!}]  
  \vskip.5truecm
 
 \noindent 
[{\it Proof of Corollary \ref{R1}}]

\noindent
  As a regular local ring is quasi-unmixed, the regular scheme $Y$ is locally quasi-unmixed.
  On the other hand, a 2-dimensional regular local ring is pseudo-rational (\cite[\S 4]{lt}),
  which yields $Y$ is pseudo-rational in codimension two.
  It is obvious that $Y$ satisfies $S_3$.
  Therefore $Y$ satisfies the all conditions of Theorem \ref{R!}, which gives the required vanishing.

The following is a well-known example that $H^1(X,\oo_X)=\CC$ for the minimal resolution 
$\varphi: X\to Y=\spec R$ for which the conditions of Theorem \ref{R!} do not hold.
We can show that the ideal which gives $\varphi$ is normal.
We also have the depth of the Veronese rings of the Rees algebra.

  \begin{ex}
Let $R=\mathbb{C}[[x,y,z]]/(x^3+y^3+z^3)$. 
Then, the following hold:
\begin{enumerate}
\item $R$ is a normal domain with $\depth R= \dim R=2$. 
\item $\m=(x,y,z)R$ is normal. 
\item $\big[H^2_{\mathcal{M}} (\mathcal{R}(\m^n))\big]_0 \ne 0$ for every integer $n \ge 1$. In particular, 
$$\depth_{\mathcal{M}} \mathcal{R}(\m^n)=2\ \ \mbox{for\ every\ }n\geq 1.$$  
\end{enumerate}

\end{ex}

\section{Non-vanishing of $R^2$}

\noindent
An example of non-vanishing $R^2f_*\oo_X\neq 0$ for 3-dimensional case was known in the last century  for normal and 
non-CM $X$.
\begin{ex}[S.D. Cutkosky {\cite[Section 3]{cut}}]\label{cut}
  There exists a projective birational morphism $f: Z\to \AA_\CC^3$ with normal $Z$, isomorphic outside the origin $0\in \AA_\CC^3$ such that 
  $$R^2f_*\oo_Z\neq 0.$$
  It is constructed as follows:
  \newline
  \noindent
  First, blowup $X_1\to \AA_\CC^3$ at the origin $\{0\}$, then take an elliptic curve $C$
   in the exceptional divisor $E_1\simeq \PP_\CC^2$.
   Then, blowup $X_2\to X_1$ at appropriate 12 points in $C$. 
   Denote the proper transform of $C$ in $X_2$ by $C'$.
   Then, there exists a base point free linear system on $X_2$ which gives a projective
   birational morphism $\psi: X_2\to Z$ which contracts $C'$ to a normal point in $Z$.
  The point $P=\psi(C')$ is not a CM point of $Z$. 
  Let $f:Z\to \AA_\CC^3$ be the canonical morphism.
  Then, by the Leray's spectral sequence, we obtain $R^2f_*\oo_Z\neq 0.$
  
\end{ex}

The following is an example that a projective birational morphism $\varphi:X\to \AA_\CC^3$
such that $X$ is normal and CM but $R^2\varphi_*\oo_X\neq0$.
\begin{ex}[L. Ma]\label{ma}
Let $f: Z \to \AA_\CC^3$
 be Cutkosky's example \ref{cut}. 
  Let $$g: X\to Z$$ be a projective birational morphism such that $X$ is arithmetically
normal and arithmetically Cohen-Macaulay as in \cite[Corollary 1.4]{b1}
 (in particular, $X$ is
normal and Cohen-Macaulay).
We first claim that $$Rg_*\oo_X = \oo_Z$$
This is local on $Z $ so we may assume $Z = \spec(R)$ for a local ring$R$
and $Y = \proj(\R(I))$ for some $ I \subset R$ such that $\R(I)$ is Cohen-Macaulay (which
comes from the arithmetically Cohen-Macaulay assumption on $X$ ), and the claim follows
from \cite[Theorem 4.1]{lip}.

Now consider the projective birational morphism $\varphi:=f\circ g: X\to \AA_\CC^3$.
 We have $X$ is normal
and Cohen-Macaulay, and $R(f\circ g)_*\oo_X = Rf_*Rg_*\oo_X = Rf_*\oo_Z\neq \oo_{\AA_\CC^3}$.
Concretely, $R^2(f\circ g)_*\oo_X \neq 0$.
\end{ex}

\vskip.5truecm
Before stepping into the construction of our example, 
we explain briefly about Brodmann's Macaulayfications \cite[Corollary 1.4]{b1}  used in Example \ref{ma},
and also about \cite[Corollary 6.12]{b2} which will be used in our example.

The Corollary 1.4 in\cite{b1} studies a three-dimensional quasi-projective variety $V$ over an algebraically closed field.
Under the assumption that $V$ has only finite number of isolated non-CM points, the corollary states
that there exists a reduced (possibly non-irreducible) curve passing through these points such that 
the blowup at the curve gives arithmetically normal and arithmetically CMfication of $V$.
This process preserves normality and regularity, {\it i.e.,} any point of the fiber of a normal (resp. regular) point is also normal 
(resp. regular).

The Corollary 6.12 in \cite{b2} study an arbitrary dimensional local ring $R$.
Under the assumption that $\spec R$ is CM on $\spec R\setminus \{\m\}$, the Corollary 6.12 states
that there exists a curve passing through the point $\m$ such that the blowup at the curve gives arithmetically normal and arithmetically CMfication of $\spec R$.
This process also preserves normality and regularity.

\vskip1truecm

\noindent
{\bf [Construction and the proof of Example \ref{R2}]} 
We  use  Example \ref{ma} and denote $Z$ and $X$ in Example \ref{ma} by $Z^{(3)}$ and $X^{(3)}$, respectively to clarify their dimensions. 
Let  $f^{(3)}:Z^{(3)}\to \AA_\CC^3$ be the Cutkosky's morphism.
We also denote the blowup $X^{(3)}\to Z^{(3)}  $ at $C\subset Z^{(3)}$ by $g^{(3)}$ and the composite $f^{(3)}\circ g^{(3)}:X^{(3)}\to \AA_\CC^3$ by $\varphi^{(3)}$.
Then, $\varphi^{(3)}$ is isomorphic outside $0$ if and only if $C\subset {f^{(3)}}^{-1}(0)$,
and otherwise $\varphi^{(3)}$ is isomorphic outside a curve in $\AA_\CC^3$.
As $X^{(3)}$ is projective birational over $\AA_\CC^3$, there exists an ideal $I\subset \CC[x_1,x_2,x_3]$
such that 
$$X^{(3)}=\proj \R(I).$$

Corollary 1.4 in \cite{b1} guarantees the existence of a curve $C\subset Z^{(3)}$ whose blowup gives  $X^{(3)}$.
However, it is not clear if the curve is in the fiber ${f^{(3)}}^{-1}(0)$ of the origin $0\in \AA_\CC^3$ or not.

So, the following two cases can happen, where we note that $f^{(3)}$ is isomorphic outside the 
origin $0\in \AA_\CC^3$ and 
\begin{enumerate}
\item[(a)] The ideal $I$ is a primary ideal associated to $(x_1,x_2, x_3)$, in case $C\subset {f^{(3)}}^{-1}(0)$, and
\item[(b)] The ideal $I$ has  height $2$ associated primes which define a smooth (possibly  non-irreducible) curve on $\AA_\CC^3\setminus 
\{0\}$, in case $C\not\subset {f^{(3)}}^{-1}(0)$.
\end{enumerate}
Here, about (a), we note that an ideal $I$ whose blowup gives an isomorphism outside a point is locally principal there
in general.
However, as we are working on UFD $\CC[x_1,x_2,x_3]$ we can replace $I$ by a maximal primary ideal.

Starting with the three-dimensional model $\varphi^{(3)}:X^{(3)}=\proj \R(I)\to \AA_\CC^3$,
we construct $N$-dimensional model $\varphi^{(N)}:X^{(N)}\to \AA_\CC^N$ by induction on the dimension $N$.
Then, we will show that $X^{(N)}$ is normal and CM, and $H^2(X^{(N)}, \oo_{X^{(N)}})\neq 0$ and also
$R^i\varphi^{(N)}\oo_{X^{(N)}} |_{\AA_\CC^N\setminus \{0\}}=0\ \ (i>0)$.

\begin{lem}\label{NtoN+1}
  Let $S_N:=\CC[x_1,\ldots, x_N]$ be the polynomial ring of $N$-variables.
  In the connection between $S_N$ and $S_{N+1}$, we denote the $N+1$-th variable $x_{N+1}$ by $z$.
  Then, $S_{N+1}=S_N[z]$ and $S_N$ is a direct summand of $S_{N+1}$.
  Let $I\subset S_N$ be a non-zero ideal and consider the Rees algebra of $(I, z)\subset S_{N+1}$
over $S_{N+1}$.
Then, we obtain the following:
\begin{enumerate}
\item[(i)] $\R(I) $ is a direct summand of $\R((I,z))$.
\item[(ii)] If $\R(I)$ is normal, then $\R((I,z))$ is also normal. 
\item[(iii)] If $X^{(N)}:=\proj\R(I)$ is normal, then for $n\gg0$, 
$Z^{(N+1)}:=\proj \R((I^n,z))$ is normal.

\end{enumerate} 
\end{lem}

\begin{proof} Every element of $\R((I,z))$ is uniquely expressed as the sum of an element without $z$ and that with $z$.
The element without $z$ is contained in $\R(I)$, this yields the proof of (i).

For the proof of (ii), it is enough to prove that $(I,z)^n$ is integrally closed by an induction on $n \ge 1$. 
\par \vspace{1mm}
Now suppose $f(z)=f_0+f_1z+\cdots +f_{\ell}z^{\ell} \in \overline{(I,z)^n}$, where $f_i \in R$. 
We first show $f_0 \in \overline{I^n}=I^n$. 
By definition, there exist $N \ge 1$ and $a_i(z) \in (I,z)^{ni}$ for $i=1,2,\ldots,N$ such that 
\begin{equation} \label{Eq-int}
f(z)^N +a_1(z)f(z)^{N-1}+\cdots+a_N(z)=0, 
\end{equation}
where $a_i(z)=a_{i0}+a_{i1}z+\cdots+a_{i\ell_i}z^{\ell_i} \in (I,z)^{ni}$ for each $i$. 
In the degree $0$ term in Eq(\ref{Eq-int}), we have 
\[
f_0^N+a_{10}f_0^{N-1}+\cdots+a_{N0}=0
\] 
For each $i$, $a_{i}(z) \in (I,z)^{ni}$ implies $a_{i0} \in I^{ni}$. 
Hence the above equation gives an integral equation of $f_0 \in \overline{I^n}=I^n$, as required. 
\par \vspace{1mm}
In particular, when $n=1$, this means $f =f_0+z(f_1+f_2z+\cdots+f_{\ell}z^{\ell-1})\in (I,z)$.
\par \vspace{1mm}
Suppose $n \ge 2$ and $\overline{(I,z)^k} = (I,z)^k$ for every $k=1,2,\ldots,n-1$. 
Then we must show $\overline{(I,z)^n}=(I,z)^n$. 
Now suppose $f(z) \in  \overline{(I,z)^n}$. 
By arguing as above, we obtain $f_0 \in I^n$. 
So we may assume $f_0=0$. 
Then we can write $f(z)=z \cdot g(z)$. 
Notice that $g(z) \in \overline{(I,z)^n} \colon z=  \overline{(I,z)^{n-1}}$. 
By assumption, we get $f(z)=z \cdot g(z) \in z(I,z)^{n-1} \subset (I,z)^n$.

(iii) By the assumption of (iii) and also by $\height I>0$, we see that
$\R(I^n)$ is normal  for $n\gg0$ (\cite[Proposition 3.6]{HM99})
Then, by (ii) $\R((I^n, z))$ is also normal for $n\gg0$,
a fortiori $Z^{(N+1)}:=\proj \R((I^n,z))$ is normal.
\end{proof}

Once we have $\varphi^{(N)}: X^{(N)}=\proj\R(I)\to \AA_\CC^N$ for $I\subset S_N$
    such that $X^{(N)}$ is normal and CM with $H^2(X^{(N)},\oo_{X^{(N)}})\neq 0$
then
we obtain the $(N+1)$-dimensional model 
$$f^{(N+1)}:Z^{(N+1)}=\proj\R((I^n,z))\to \AA_\CC^{N+1}$$ 
for $n\gg0$ such that
$$H^2(Z^{(N+1)}, \oo_{Z^{(N+1)}})\neq 0,$$
by (i) in Lemma \ref{NtoN+1}.
As $X^{(N)}$ is normal, so is $Z^{(N+1)}$  by (iii) of Lemma \ref{NtoN+1}.
However even if $X^{(N)}$ is CM, the scheme $Z^{(N+1)}$ is not necessarily CM.
We will see that $Z^{(N+1)}$ has at most one non-CM point
so that we can apply Brodmann's Macaulayfication to get $X^{(N+1)}$ which will be normal and CM
satisfying
 $$H^2(X^{(N+1)}, \oo_{X^{(N+1)}})\neq 0.$$

\begin{lem}\label{CMtoN+1}
  Let $I\subset S_N:=\CC[x_1,\ldots,x_N] $ be an ideal either of the following types:
  \begin{enumerate}
  \item[(a)]
 a maximal primary associated to $(x_1,\ldots,x_N)\subset S_N$, or
    \item [(b)]
  $\surd\overline{I}$ defines a curve which is  smooth (possibly  non-irreducible) on $\AA_\CC^N\setminus \{0\}$.
\end{enumerate}
   Assume $X^{(N)}:=\proj \R(I)$ is normal and CM such that the restriction of $\varphi^{(N)}: X^{(N)}\to
   \AA_\CC^N$ onto $\AA_\CC^N\setminus \{0\}$ is either isomorphic (in case (a)) or 
   the blowup at the smooth curve (in case (b)).
   Then, replacing $I$ by $I^n$ for $n\gg 0$, we obtain the following:
   \begin{enumerate}
\item[(i)]   $Z^{(N+1)}:=\proj \R((I, z))$ is normal and has at most one non-CM point,
\item[(ii)] There is a projective birational morphism $g^{(N+1)}: X^{(N+1)}\to Z^{(N+1)}$ such that 
$X^{(N+1)}$ is normal and CM, and $Rg^{(N+1)}_*\oo_{X^{(N+1)}}=\oo_{Z^{(N+1)}}$ holds. 
\end{enumerate}
\end{lem}

\begin{proof} Let $I$ be generated by $g_1,\ldots,g_r\in S_N=\CC[x_1,\ldots,x_N]$, then 
$(I,z)\subset S_{N+1}=\CC[x_1,\ldots,x_N, z] $ is generated by $z, g_1,\ldots,g_r$.
By using these generators, we have an open affine covering $X^{(N)}=\bigcup_{i=1}^r U_i$, where
$$U_i=\spec S_N\left[ \frac{g_1}{g_i},\ldots,\stackrel{i}\vee,\ldots  \frac{g_r}{g_i}\right].$$
In the same way, we have an open affine covering $Z^{(N+1)}=\bigcup_{i=0}^r \widetilde{U}_i$, where
$$\widetilde{U}_0=\spec S_N\left[ z, \frac{g_1}{z},\ldots\ldots  \frac{g_r}{z}\right]\ \ \mbox{and}$$
$$\begin{array}{lll}
\widetilde{U}_i&=\spec S_N\left[ z, \frac{z}{g_i},\frac{g_1}{g_i},\ldots,\stackrel{i}\vee,\ldots  \frac{g_r}{g_i}\right]&  \\
   &=\spec S_N\left[\frac{z}{g_i},\frac{g_1}{g_i},\ldots,\stackrel{i}\vee,\ldots  \frac{g_r}{g_i}\right]
&=U_i\times \AA_\CC^1\ \ (i=1,\ldots,r)\\
\end{array}$$
Then, for $i=1,\ldots,r$, $\widetilde{U}_i$ is CM since $U_i$ is CM.
Hence, the non-CM locus on $Z^{(N+1)}$ is on the closed subscheme
$$\widetilde{U}_0\setminus\bigcup_{i=1}^r\widetilde{U}_i=Z\left(\frac{g_1}{z},\ldots\ldots , \frac{g_r}{z}\right)
=:\mathcal Z,$$
which is isomorphic to
$$\spec \CC[x_1,\ldots, x_N,z]/(g_1,\ldots, g_r)=(\spec\CC[x_1,\ldots, x_N]/I )\times \AA_\CC^1.$$
Here, denote $\mathcal{Z}_0:=\spec\CC[x_1,\ldots, x_N]/I $, then $\mathcal{Z}_0$ is mapped isomorphically 
to the center of the 
blow up $\varphi^{(N)}:X^{(N)}\to \AA_\CC^N$.
Let 
$$P:=(0, 0)\in \mathcal{Z}_0\times \AA_\CC^1\subset \widetilde{U}_0,$$ 
where the first coordinate $0\in \mathcal{Z}_0$ corresponds to the origin $0\in \AA_\CC^N$ by $\varphi^{(N)}$,
and the second coordinate $0\in  \AA_\CC^1$ is the point $z=0$ in $\AA_\CC^1=\spec\CC[z]$.
Then, $P\in \widetilde{U}_0$ is the unique possible non-CM point in $Z^{(N+1)}$.
Indeed, every point $$Q\in \mathcal{Z}_0\times (\AA_\CC^1\setminus \{0\})$$ 
is a smooth point in $\widetilde{U}_0$, because $z(Q)\neq 0$,
which implies $Q $ is not in the exceptional divisor that is defined by $z=0$ on $\widetilde{U}_0$. 

 In case (a), the support of $\mathcal{Z}_0$ is one point.
Therefore, in this case, smoothness of $\widetilde{U}_0$ at 
 any point in $\mathcal{Z}_0\times (\AA_\CC^1\setminus \{0\})$ implies that 
$P\in Z^{(N+1)}$ is the unique possible non-CM point.

In case (b),  
$\widetilde{U}_0$ is smooth at
every point $Q\in (\mathcal{Z}_0\times \{0\})\setminus\{P\}$, because $Z^{(N+1)}$ is smooth on
$(\psi^{(N+1)})^{-1}\left((\AA_\CC^N\setminus\{0\})\times \AA_\CC^1\right)$ by the assumption on $I$ and $Q$ is in this
inverse image.
This completes the proof of (i).

Next we will find a center of the blowup to get a normal Macaulayfication as in \cite[Corollary 6.12]{b2}.

In case (a), let $f:=z$ and 
$$\q:=(z, \frac{g_1}{z},\ldots, \frac{g_r}{z})\subset \m_P\subset \oo_{\widetilde{U}_0,P}.$$
Then $f\in \oo_{\widetilde{U}_0,P}$ and the maximal primary ideal $\q$ satisfy the condition of \cite[Lemma 6.7]{b2},
and therefore they produce a generic complete intersection $L=(y_1,\ldots, y_N)$, where $N=\dim Z^{(N+1)}-1$
and the locus of $L$ contains a non-CM point $P\in \widetilde{U}_0$.
(The ideal generated by the sequence obtained in \cite[Lemma 6.7]{b2} is called a generic complete intersection and
the generators become a generic $pS^+$ sequence, see \cite[Remark 6.9]{b2}.)
By \cite[Corollary 6.12]{b2}, the blow-up by $\surd\overline{L}^h$ gives a normal and arithmetically 
Cohen-Macaulay for $h\gg0$. 

In case (b), there is a coordinate function, say $x_1$, of $\CC[x_1,\ldots, x_N]$ such that $x_1\not\in \surd\overline{I}$.
Indeed, otherwise $I\subset S_N$ would be maximal primary at $0$.
Let $f:=z$ and $$\q:=(z, x_1,\frac{g_1}{z},\ldots, \frac{g_r}{z})\subset \m_P\subset \oo_{\widetilde{U}_0,P}.$$ 
Then $f\in \oo_{\widetilde{U}_0,P}$ and the maximal primary ideal $\q$ satisfy the condition of \cite[Lemma 6.7]{b2},
and therefore we also obtain a generic complete intersection $L=(y_1,\ldots, y_N)$.
In the same way as in (a), the blow-up by $\surd\overline{L}^h$ gives a normal and arithmetically 
Cohen-Macaulay for $h\gg0$.  

Now in both cases (a) and (b), take the blowup $g^{(N+1)}: X^{(N+1)}\to Z^{(N+1)}$ by the ideal $\surd\overline{L}^h$, then 
$X^{(N+1)}$ is normal and arithmetically CM.
Here, by \cite[Theorem 4.1]{lip}, it follows $R\varphi^{(N+1)}_*\oo_{X^{(N+1)}}=\oo_{Z^{(N+1})}$.
\end{proof}

By  Lemma \ref{NtoN+1} and  \ref{CMtoN+1}, 
 $H^2(X^{(N+1)},\oo_{X^{(N+1)}})\neq 0$ holds if $ H^2(X^{(N)},\oo_{X^{(N)}})\neq 0$.
We know  $ H^2(X^{(3)},\oo_{X^{(3)}})\neq 0$.
Then, by induction on the dimension $N$, we obtain  for every $N\geq 3$
$$\varphi^{(N)}:X^{(N)}\to \AA_\CC^N$$
such that  $X^{(N)}$ is normal and CM, and satisfies
    $$R^i\varphi_*^{(N)}\oo_{X^{(N)}}|_{\AA_\CC^N\setminus\{0\}}=0,\ \ (i>0), \ \mbox{and}\ 
        H^2(X^{(N)},\oo_{X^{(N)}})\neq 0.$$

\begin{rem}
Let $k$ a field of characteristic not $3$ and 
put $R=k[x,y,z]_{(x,y,z)}$ and $\m=(x,y,z)$. 
Then Huckaba and Huneke \cite[Theorem 3.11]{hh} gave an 
$\m$-primary normal ideal $I \subset R$ with $H^2(X,\oo_X) \ne 0$, 
where $X=\proj \R(I)$.
$X$ is known to be non-CM, however the dimension of non-CM locus is not known.
If $k$ is an algebraically closed field and the non-CM locus of $X$ consists of  finite closed points,
then by starting from this, one can obtain a similar example as Example 1.5. 
\end{rem}


\end{document}